\documentclass[draft,12pt,onecolumn]{IEEEtran}

\usepackage{amsmath}
\usepackage{amsthm}
\usepackage{amsfonts}
\usepackage{amssymb}
\usepackage{mathrsfs}
\usepackage{soul}
\usepackage[colorlinks = true, allcolors=black, urlcolor=darkgray]{hyperref}
\usepackage{url}
\usepackage[usenames,dvipsnames]{xcolor}
\usepackage{enumerate}
\usepackage{tikz}
\usetikzlibrary{matrix}

\usepackage{algorithm}
\usepackage{algorithmic}

  \newcommand\phantomarrow[2]{%
  \setbox0=\hbox{$\displaystyle #1\to$}%
  \hbox to \wd0{%
    $#2\mapstochar
     \cleaders\hbox{$\mkern-1mu\relbar\mkern-3mu$}\hfill
     \mkern-7mu\rightarrow$}%
  \,}

\usepackage{datetime}

\newcommand{\Z}{\mathbb{Z}}
\newcommand{\R}{\mathbb{R}}
\newcommand{\C}{\mathbb{C}}

\newcommand{\T} {\mathbb{T}}

\newcommand{\M} {\mathcal{M}}
\newcommand{\N} {\mathcal{N}}
\newcommand{\KL} {\mathbb{S}}
\newcommand{\J} {\mathbb{J}}
\newcommand{\Lp} {\mathcal{L}}
\newcommand{\de} {\mathrm{d}}
\newcommand{\tr} {\mathrm{tr}}

\renewcommand{\check}{\color{black}}

\theoremstyle{definition}
\newtheorem{Lemma}{Lemma}

\newtheorem{Theorem}{Theorem}
\newtheorem{Remark}{Remark}

\newtheorem{Proposition}{Proposition}
\newtheorem{Problem}{Problem}

\definecolor{Royalblue}{cmyk}{1,0.30,0.2,0.2}

\newcommand{\beq}{\begin{equation}}
\newcommand{\eeq}{\end{equation}}
\newcommand{\bea}{\begin{eqnarray}}
\newcommand{\eea}{\end{eqnarray}}

\def\bmat{\left[ \begin{array}}
\def\emat{\end{array} \right]}

\author{Giacomo Baggio
\thanks{G.~Baggio is with the Dipartimento di Ingegneria dell'€™Informazione, Universit\`a di Padova, via Gradenigo, 6/B€" I-35131 Padova, Italy. E-mail: \texttt{giacomo.baggio@studenti.unipd.it}.}}

\begin{document}

\title{Further Results on the Convergence \\ of  the Pavon--Ferrante Algorithm \\ for Spectral Estimation}

\maketitle

\begin{abstract} 
In this paper, we provide a detailed analysis of the global convergence properties of an extensively studied and extremely effective fixed-point algorithm for the Kullback--Leibler approximation of spectral densities, proposed by Pavon and Ferrante in \cite{PF06}. Our main result states that the algorithm globally converges to one of its fixed points.
\end{abstract}


\section{Introduction}

In recent years, the problem of approximating---in an optimal sense---a spectral density with another one satisfying some given constraints has received considerable attention in the control and signal processing community. In \cite{GL03}, Georgiou and Lindquist considered the following formulation of the above-mentioned approximation problem: Find the optimal Kullback--Leibler approximation of a spectral density given
\begin{enumerate}[(i)]
\item  an a-priori estimate of the spectrum describing a zero-mean second-order stationary process, and
\item asymptotic state-covariance data that are typically inconsistent with the a-priori estimate. These data are obtained by feeding the above process to a measurement device consisting of a bank of rational filters. 
\end{enumerate}
According to this formulation, the approximation problem turns into a (convex) optimization problem with integral constraints which falls into the celebrated category of \emph{(generalized) moment problems}. During the past century, the latter class of problems has played a crucial role in many ares of the engineering and mathematical sciences, see e.g. \cite{A65moments,BL08}
and references therein. In particular, some noteworthy (generalized) moment problems, which are close relatives of the Georgiou--Lindquist approximation problem, are the \emph{covariance extension problem} \cite[Ch. 12.5]{LP15}, \cite{G87,BLGM95,BGL98},  
\emph{THREE-like spectral estimation} \cite{BGL00,G01,RFP09,RFP10}, 
the classical \emph{Nevanlinna--Pick interpolation problem} \cite{BTL01,BLN03,BTLM06}, and its generalization, the \emph{augmented Basic Interpolation Problem} (aBIP) \cite{dym1998basic,bolotnikov2006boundary}. 
Moreover, variations of the latter classes of problems have generated a huge stream of literature in recent years, see for instance \cite{FPR08,GL08,CFPP11,FMP12,ZF12,FPZ12,Z14,Z15,GL16}. Among the large number of applications emerging from the class of (generalized) moment problems it is worth mentioning, besides spectral estimation, those related to system modeling/identification and optimal $\mathcal{H}_{\infty}$~control  \cite{stefanovski2015,stefanovski2015reformulation,stefanovski2016new}.

In \cite{GL03}, the optimization problem was approached by resorting to the dual problem which is finite-dimensional but, in general, it does not admit a closed-form solution. Also, standard gradient-based minimization techniques for the numerical solution of the dual problem have proved to be computationally demanding and severely ill-conditioned. In order to tackle these issues, in \cite{PF06}, Pavon and Ferrante proposed an alternative iterative method for the solution of the Georgiou--Lindquist approximation problem. The Pavon--Ferrante algorithm is a surprisingly simple and efficient nonlinear fixed-point iteration in the set of positive semi-definite unit trace matrices. Furthermore, the algorithm exhibits very attractive and robust properties from a numerical viewpoint, since it can be implemented via the solution of an algebraic Riccati equation and a Lyapunov equation \cite{FRT11}. On the other hand, despite the huge amount of numerical evidences, proving the convergence of the latter algorithm to a prescribed set of fixed points---which provide the solution of the approximation problem---has revealed to be an highly non-trivial challenge \cite{FPR07,FRT11}. In particular, in \cite{FRT11} it has been shown that the Pavon--Ferrante algorithm is locally convergent to the aforesaid set of fixed points, through a rather tortuous yet enlightening proof. Nonetheless, a proof of \emph{global} convergence of the algorithm, though conjectured and supported by a large number of numerical simulations, has so far been elusive.

The present paper addresses this problem. Specifically, we consider a cost functional arising from the formulation of the dual problem and we show that the latter is decreasing along the trajectories generated by the Pavon--Ferrante algorithm. This provides an answer to a conjecture raised in \cite[Sec.~V]{FRT11} and leads to the main contribution of the paper, namely, a proof of {\em global} convergence of the Pavon--Ferrante algorithm towards its set of fixed points.


{\em Paper structure.} The paper is organized as follows. In Section \ref{sec:prob}, we introduce the Georgiou--Lindquist spectral approximation problem and its solution via the Pavon--Ferrante algorithm. In Section \ref{sec:prelim}, we estabilish some auxiliary results. Section \ref{sec:global} contains the proof of global convergence to the set of fixed points of the Pavon--Ferrante algorithm. 
Finally, Section \ref{sec:conclusion} collects some concluding remarks.

{\em Notation.} We let $\Z$, $\C$, $\R$, $\R_{>0}$, $\R_{\geq 0}$, and $\T$ denote the set of integer, complex, real, positive real, non-negative real numbers, and the unit circle in the complex plane, respectively. We denote by $\C^{n\times m}$ the set of $n\times m$ matrices with complex entries. Given $A\in\C^{n\times m}$, $A^{*}$ will denote the Hermitian transpose of $A$.  We write $A\geq 0$ ($A>0$) to mean that $A=A^{*}\in\C^{n\times n}$ is positive semi-definite (positive definite, respectively). For $A\geq 0$, $A^{1/2}$ will denote the principal matrix square root of $A$, i.e., the unique positive semi-definite Hermitian matrix whose square is $A$. We endow the space $\C^{n}$ with the standard inner product $\langle x,y\rangle=x^{*}y$ and norm $\|x\|^{2}:=\langle x,x\rangle$, for $x,y\in\C^{n}$, and the space of Hermitian matrices of dimension $n\times n$, denoted by $\mathbb{H}_{n}$, with the trace inner product $\langle X,Y\rangle:=\mathrm{tr}(XY^{*})$ and Frobenius norm $\|X\|_{\mathrm{F}}^{2}:=\mathrm{tr}(XX^{*})$, for $X,Y\in \mathbb{H}_{n}$, where $\tr(\cdot)$ denotes the trace operator. Moreover, we denote by $\mathfrak{S}_{n}$ the convex set of $n\times n$ positive semi-definite unit trace matrices. For a matrix-valued function $G(z)$ in the complex variable $z$, $G^{*}(z)$ will denote the analytic continuation of the function that for $z\in\T$ equals the Hermitian transpose of $G(z)$.
Finally, $\mathcal{C}(\T)$ will denote the set of continuous functions on $\T$, and $\mathcal{C}_{+}(\T)$ the set of continuous functions on $\T$ which take (strictly) positive values on the same region, i.e., the space of continuous coercive discrete-time spectral density functions. 

\section{Problem statement}\label{sec:prob}

In this section, we first review the spectral density approximation problem treated in \cite{GL03}. Then, we discuss its solution via the algorithm proposed in \cite{PF06}. 

\subsection{The Georgiou--Lindquist approximation problem}

Let $\{\,y(t),\, t\in\Z\,\}$ be a zero-mean purely nondeterministic second-order stationary process and assume that an a-priori estimate  $\Psi(e^{j\theta})\in \mathcal{C}_{+}(\T)$ of the spectral density of $y(t)$ is given.
Consider the rational matrix transfer function
\[
	G(z) = (zI-A)^{-1}B,\quad A\in\C^{n\times n},\ \ B\in\C^{n\times 1},
\]
of the discrete-time system
\[
	x(t+1) = Ax(t) + By(t),\quad t\in\Z,
\]
where $A$ is Schur stable, i.e., all the eigenvalues of $A$ are strictly inside $\T$, and the pair $(A,B)$ is reachable. Note that the $n$-dimensional process $x(t):=[x_{1}(t),x_{2}(t),\dots,x_{n}(t)]^{\top}$ coincides with the output of a bank of $n$ filters $G(z):=[g_{1}(z),g_{2}(z),\dots,g_{n}(z)]^{\top}$ fed by $y(t)$.

Suppose we know the steady-state covariance of the process $x(t)$, which we denote by $\Sigma>0$.\footnote{On the problem of estimating covariance matrices from measurements obtained by linear filtering see also \cite{ZF12,FPZ12}.}
Given the estimate $\Psi(e^{j\theta})$ and the steady-state covariance $\Sigma$, the task is to estimate the spectral density of the process $y(t)$. To this end, we need to find the spectral density $\hat\Phi(e^{j\theta})\in\mathcal{C}_{+}(\T)$ which is the ``closest possible'', in a suitable sense, to the a-priori estimate $\Psi$ among all spectra $\Phi(e^{j\theta})\in\mathcal{C}_{+}(\T)$ satisfying the constraint
\[
	\int_{-\pi}^{\pi}G(e^{j\theta})\Phi(e^{j\theta})G^{*}(e^{j\theta})\frac{\de \theta}{2 \pi}=\int G_\theta\Phi_\theta G^{*}_\theta =\Sigma.
\]
(In order to lighten the notation, throughout the paper we let $G_{\theta}:=G(e^{j\theta})$, $\Phi_{\theta}:=\Phi(e^{j\theta})$, $\Psi_{\theta}:=\Psi(e^{j\theta})$, and for integrals we use the above shorthand, where the integration takes place on the unit circle w.r.t. the normalized Lebesgue measure.)
By using the Kullback--Leibler divergence \cite{CJ12} as ``measure of closeness''  between spectral densities, namely,
\[
	\KL(\Phi_\theta\| \Psi_\theta):=\int\Psi_\theta\log\frac{\Psi_\theta}{\Phi_\theta},
\] 
the problem can be formally stated as follows.

\begin{Problem}[Georgiou--Lindquist approximation problem \cite{GL03}] \label{prob1}
Let $\Psi_\theta\in \mathcal{C}_{+}(\T)$ and $\Sigma\in\mathbb{H}_{n}$, $\Sigma>0$. Find $\hat\Phi_\theta \in\mathcal{C}_{+}(\T)$ that solves
\[
	\min_{\Phi_\theta\in\mathcal{K}} \ \KL(\Phi_\theta\| \Psi_\theta),
\]
where 
\begin{align}\label{eq:setC}
	\mathcal{K}:=\left\{\,\Phi_\theta\in \mathcal{C}_{+}(\T) \, :\, \int G_\theta\Phi_\theta G^{*}_\theta =\Sigma\,\right\}.
      \end{align}
\end{Problem}

The variational analysis outlined in \cite{GL03} (see also \cite{PF06,FRT11} where some additional details are spelled out and \cite{G01,G02} for the existence part) leads to the following result.

\begin{Theorem}
The set $\mathcal{K}$ as defined in \eqref{eq:setC} is non-empty if and only if there exists $H\in\C^{1\times n}$ s.t.
\begin{align}\label{eq:feasab}
	\Sigma -A\Sigma A^{*} = BH+H^{*} B^{*}. 
\end{align}
Moreover, assuming that the above condition is fulfilled, there exists $\hat{\Lambda}\in\mathbb{H}_{n}$ such that
\begin{align}
	& G^{*}_\theta \hat{\Lambda} G_\theta >0, \quad \forall\,\theta\in[-\pi,\pi),\label{cond1}\\
	& \int G_\theta\frac{\Psi_\theta}{G^{*}_\theta\hat{\Lambda} G_\theta}G^{*}_\theta =\Sigma.\label{cond2}
\end{align}
For any such $\hat{\Lambda}$,
\begin{align}\label{eq:spest}
	\hat{\Phi}_\theta:=\frac{\Psi_\theta}{G^{*}_\theta\hat{\Lambda} G_\theta}
\end{align}
is the unique solution of Problem \ref{prob1}.
\end{Theorem}

Supposing the feasibility condition \eqref{eq:feasab} satisfied, in view of the above theorem, Problem \ref{prob1} can be reduced to the problem of finding $\Lambda\in\mathbb{H}_{n}$ satisfying conditions \eqref{cond1}-\eqref{cond2}.

In \cite{GL03}, Georgiou and Lindquist exploited duality theory to arrive at the equivalent convex optimization problem
\begin{align}\label{eq:GLprob}
\min_{\Lambda\in\Lp}\ \J(\Lambda),
\end{align}
where
\begin{align}\label{eq:PFcostold}
	\J\colon &\Lp\to\R\notag \\
	 &\phantomarrow{\Lp}{\Lambda} \tr(\Lambda\Sigma)-\int \Psi_\theta \log G^{*}_\theta \Lambda G_\theta,
\end{align}
and 
\begin{align*}
	&\Lp:=\{\,\Lambda\in\mathbb{H}_{n} \, :\, G^{*}_\theta\Lambda G_\theta>0, \ \forall \theta\in[-\pi,\pi)\,\}.
\end{align*}
In \cite[Thm.~5]{GL03} it has been established that the above dual problem admits a unique solution $\hat{\Lambda}$ on $\Lp(\Gamma):=\Lp\,\cap\, \mathrm{Range}\,\Gamma$,\footnote{It is worth observing however that Problem \eqref{eq:GLprob} has in general infinitely many solutions on $\Lp$.} where $\Gamma$ is the linear operator defined by  
\begin{align}\label{eq:Gamma}
	\Gamma\colon &\mathcal{C}(\T)\to\mathbb{H}_{n}\notag \\
	 &\phantomarrow{\mathcal{C}(\T)}{\Phi_\theta} \int G_\theta \Phi_\theta G^*_\theta,
\end{align}
and such a solution satisfies \eqref{cond2}, so that $\hat{\Lambda}$ returns the optimal estimate $\hat{\Phi}_\theta$ via \eqref{eq:spest}. 
After a suitable parametrization of $\mathcal{L}(\Gamma)$, problem \eqref{eq:GLprob} can be numerically solved using Newton-like minimization methods \cite[Sec.~VII]{GL03}. However, as mentioned in the introduction, these techniques are affected by several numerical issues related to unboundedness of the gradient of $\J(\cdot)$ around the neighborhood of the boundary and high computational burden due to a large number of backstepping iterations \cite[Sec.~VII]{GL03}, \cite{PF06,FPR07,FRT11}. 

%


\subsection{The Pavon--Ferrante algorithm}\label{subsec:PF}

An alternative, numerically robust algorithm for the solution of Problem \ref{prob1} has been proposed by Pavon and Ferrante in \cite{PF06} and further investigated in \cite{FPR07,FRT11}. Before presenting the algorithm, we introduce some simplifications in the formulation of Problem \ref{prob1}; namely, we suppose
\begin{enumerate}[(i)]
\item the a-priori estimate $\Psi_\theta$ to be such that $\int\Psi_\theta=1$, and
\item the steady-state covariance $\Sigma$ to be normalized to identity, i.e., $\Sigma=I$.
\end{enumerate}
These assumptions can be made without any loss of generality, as explained in \cite[Remark 2.3]{FRT11}. Furthermore, notice that, with these simplifications in place, the cost functional $\J(\cdot)$ in \eqref{eq:PFcostold} becomes
\begin{align}\label{eq:PFcost}
	\J(\Lambda)&=\tr(\Lambda)-\int \Psi_\theta \log G^{*}_\theta \Lambda G_\theta.
\end{align}
The Pavon--Ferrante algorithm is a fixed-point iteration of the form
\begin{align}\label{eq:PFalg}
	&\Lambda_{k+1}=\Theta(\Lambda_{k}):= \int \Lambda_{k}^{1/2}G_\theta\left(\frac{\Psi_\theta}{G^{*}_\theta\Lambda_{k} G_\theta}\right)G^{*}_\theta\Lambda_{k}^{1/2},
\end{align}
for $k\in\Z$, $k\geq 0$, where the initialization is taken to be a positive definite trace-one matrix $\Lambda_{0}>0$. Iteration \eqref{eq:PFalg} features several remarkable properties. Firstly, it preserves unit trace and positivity. Furthermore, by introducing the sets
\begin{align}\label{eq:L+}
	\mathcal{M} &:=\{\,\Lambda\in\mathfrak{S}_{n} \, :\, G^{*}_\theta\Lambda G_\theta>0,\ \forall \theta\in[-\pi,\pi)\,\},\\
	\mathcal{M}_+ &:=\{\,\Lambda\in \mathcal{M}\,:\, \Lambda>0\}\subset \M,
\end{align}
it has been shown in \cite[Thm. 4.1]{PF06} that $\Theta(\cdot)$ maps elements of $\M$ ($\M_+$) into elements of $\M$ ($\M_+$, respectively).

Secondly, and most importantly, if iteration \eqref{eq:PFalg} converges to a \emph{positive definite} fixed point of $\Theta(\cdot)$, say $\hat{\Lambda}>0$, then $G^{*}_\theta \hat{\Lambda} G_\theta >0$, $\forall\,\theta\in[-\pi,\pi)$, and, by multiplying Eq.~\eqref{eq:PFalg} on both sides by $\hat{\Lambda}^{-1/2}$, 
\[
	\int G_\theta\frac{\Psi_\theta}{G^{*}_\theta\hat{\Lambda} G_\theta}G^{*}_\theta = I,
\]
so that conditions \eqref{cond1}-\eqref{cond2} are satisfied. As a consequence, if the feasibility condition in \eqref{eq:feasab} is satisfied by $\Sigma=I$, such a $\hat{\Lambda}$ yields the solution of Problem \ref{prob1} via \eqref{eq:spest}.
Importantly, such a fixed point always exists. In fact, let $\mathcal{S}$ denote the space of $\Lambda\in\mathbb{H}_{n}$ satisfying \eqref{cond1}-\eqref{cond2}, in \cite[Thm.~3.2]{FRT11} it has been shown that the set of positive definite fixed points of iteration \eqref{eq:PFalg} defines a non-empty open convex set $\mathcal{P}$ of the space $\mathcal{S}$. To conclude, we remark that the positive definite ones are not the only fixed points of iteration \eqref{eq:PFalg} which provide the solution to Problem \ref{prob1}. Indeed, in the closure of $\mathcal{P}$ there exist singular elements which still satisfy \eqref{cond1}-\eqref{cond2} and, thus, solve Problem \ref{prob1} via \eqref{eq:spest}, see \cite[Sec.~II-B]{FRT11}. In general, however, singular fixed points are not guaranteed to satisfy conditions \eqref{cond1}-\eqref{cond2} and, therefore, to solve Problem \ref{prob1} via \eqref{eq:spest}.

\section{Preliminary results}\label{sec:prelim}

In this section, we collect some auxiliary results which will be used in the proof of the main theorems presented in the next section. The first result is a consequence of Jensen's inequality \cite[Thm. 2.6.2]{CJ12}.
\begin{Lemma}
\label{lem1}
Let $X\subseteq \R$. Consider an integrable function $f\colon X\to \R_{>0}$ and an integrable function $w\colon X\to \R_{>0}$ satisfying $\int_{X}w(x)\,\de x=1$, then
\[
	\log\int_{X} w(x) f(x)\, \de x\ge \int_{X} w(x) \log f(x)\,\de x,
\]
and the equality is attained if and only if $f(x)$ is constant a.e. on $X$.
\end{Lemma}
%


Another ancillary lemma is stated and proved below.

\begin{Lemma}\label{lem2}
Let $X\subseteq \R$. Let $w\colon X\to \R$ and $f \colon X\times X\to\R_{\geq 0}$ be integrable functions with $f(\cdot,\cdot)$ symmetric, i.e. $f(x,y)=f(y,x)$ for all $x,y\in X$. Then it holds
\begin{align}\label{eq:lem2}
\int_{X} \int_{X}w(x)^{2} f(x,y)\, \de x\, \de y\geq \int_{X} \int_{X} w(x)w(y) f(x,y)\, \de x\,\de y.
\end{align}
\end{Lemma}

\begin{proof}
Thanks to the symmetry of $f(\cdot,\cdot)$, we have
\begin{align}
& \int_{X} \int_{X}w(x)^{2} f(x,y)\, \de x\,\de y =\int_{X} \int_{X}\left(\frac{w(x)^{2}+w(y)^{2}}{2}\right) f(x,y)\, \de x\,\de y.\notag
\end{align}
Now, since $w(x)^{2}+w(y)^{2}\geq 2w(x)w(y)$, due to the fact that $(w(x)-w(y))^{2}\geq 0$, for all $x,y\in X$, the claim follows.
\end{proof}

{\check

Next we focus the attention on the function $\J(\cdot)$, as defined in Eq.~\eqref{eq:PFcost}. This function will play a key role in the convergence analysis presented in Section \ref{sec:global}. 
\begin{Lemma}\label{lem:cont}
$\J(\cdot)$ is a continuous and bounded function on $\mathfrak{S}_{n}$.
\end{Lemma}
\begin{proof}
We first note that, in view of the stability of $A$ and reachability of the pair $(A,B)$, for all $\Lambda\in\mathfrak{S}_{n}$, $G^{*}_{\theta}\Lambda G_{\theta}$ is a nonzero rational spectral density analytic on (an open annulus containing) $\T$. This in turn implies that $\log G^{*}_{\theta}\Lambda G_{\theta}$ is integrable on $\T$, see e.g. \cite[p.~64]{rozanov1967stationary}. Since $\Psi_{\theta}$ is bounded on $\T$, $\Psi_{\theta}\log G^{*}_{\theta}\Lambda G_{\theta}$ is again integrable on $\T$. This in turn implies that $\J(\cdot)$ is bounded on $\mathfrak{S}_{n}$. Now let $\bar{\Lambda}\in\mathfrak{S}_{n}$ and consider any sequence $\{\Lambda_{k}\}_{k\ge 0}$, in $\mathfrak{S}_{n}$ such that $\lim_{k\to\infty}\Lambda_{k}=\bar{\Lambda}$. 
 The corresponding sequence $\{G^{*}_{\theta}\Lambda_{k} G_{\theta}\}_{k\ge 0}$ is composed of nonzero rational spectral densities analytic on $\T$ and such that $\lim_{k\to\infty}G^{*}_{\theta}\Lambda_{k} G_{\theta}=G^{*}_{\theta}\bar{\Lambda} G_{\theta}$ uniformly on $\T$, where the limit $G^{*}_{\theta}\bar{\Lambda} G_{\theta}$ is a spectral density as before. Hence, from \cite[Cor.~4.6]{nurdin2006spectral}, it follows that the sequence $\{\log G^{*}_{\theta}\Lambda_{k} G_{\theta}\}_{k\ge 0}$ is uniformly integrable on $\T$. Eventually, since $\Psi_{\theta}$ is bounded on $\T$, $\{ \Psi_\theta\log G^{*}_{\theta}\Lambda_{k} G_{\theta}\}_{k\ge 0}$ is again uniformly integrable, so that by Vitali's convergence theorem \cite[p.~133]{rudin1987real} it holds
\begin{align*}
	\lim_{k\to\infty} \J(\Lambda_{k})&=1-\lim_{k\to\infty}\int \Psi_\theta \log G^{*}_\theta \Lambda_{k} G_\theta\\
	& = 1-\int \lim_{k\to\infty} \Psi_\theta \log G^{*}_\theta \Lambda_{k} G_\theta\\
	& = 1-\int\Psi_\theta \log G^{*}_\theta \bar{\Lambda} G_\theta = \J(\bar{\Lambda}),
\end{align*}
which proves continuity.
\end{proof}



Consider the orthogonal complement (w.r.t. the trace inner product in $\mathbb{H}_{n}$) of $\mathrm{Range}\,\Gamma$, where the linear operator $\Gamma$ has been defined in \eqref{eq:Gamma}. This quantity has been shown in \cite[Sec.~IV-A]{FPR08} to be given by
\begin{align}\label{eq:rangeG}
	&(\mathrm{Range}\,\Gamma)^{\perp} = \{\,X\in\mathbb{H}_{n}\,:\, G^{*}_{\theta}XG_{\theta}=0,\ \forall \theta\in[-\pi,\pi)\,\}.
\end{align}
Given $\Lambda\in\mathfrak{S}_{n}$ and any non-zero matrix $\Lambda^{\perp}\in(\mathrm{Range}\,\Gamma)^{\perp}$ such that $\Lambda+\Lambda^{\perp}\geq 0$, we observe that $\J(\Lambda) = \J(\Lambda+\Lambda^{\perp})$ and $\Lambda+\Lambda^{\perp}\in \mathfrak{S}_{n}$, since every element in $(\mathrm{Range}\,\Gamma)^{\perp}$ is traceless \cite[Sec. II]{FRT11}.

At this point, we analyze the behavior of $\J(\cdot)$ in the region of the boundary of $\mathfrak{S}_{n}$ defined by 
\begin{align}
	\N:=\{\,\Lambda\in\mathfrak{S}_{n} \, :\, \exists\, \bar{\theta}\in[-\pi,\pi) \text{ s.t. } G^{*}_{\bar{\theta}}\Lambda G_{\bar{\theta}}=0\,\}.
\end{align}
The following lemma provides a useful result in this regard.
\begin{Lemma}\label{lem:infcomp}
Suppose $\Psi_{\theta}\in\mathcal{C}_{+}(\T)$. For all $\bar{\Lambda}\in\N$, the (right-sided) directional derivative of $\J(\cdot)$ at $\bar{\Lambda}$ along any direction $\Delta\bar{\Lambda}\in\mathbb{H}_{n}$ such that $\bar\Lambda+\Delta\bar{\Lambda}\in\M$ takes the value $-\infty$.
\end{Lemma}
\begin{proof}
Let $\bar{\Lambda}\in\N$. First, we note that $\bar{\Lambda}+\varepsilon\Delta\bar{\Lambda}\in\M$ for all $\Delta\bar{\Lambda}\in\mathbb{H}_{n}$ such that $\bar{\Lambda}+\Delta\bar{\Lambda}\in\M$, and for all $\varepsilon\in(0,1]$. 
The (right-sided) directional derivative of $\J(\cdot)$ at $\bar{\Lambda}$ in the direction $\Delta\bar{\Lambda}$ is given by
\begin{align}
	\nabla \J(\bar{\Lambda};\bar{\Lambda}+\Delta\bar{\Lambda})&:=\lim_{\varepsilon\to0^{+}} \frac{\J(\bar{\Lambda}+\varepsilon \Delta\bar{\Lambda})-\J(\bar{\Lambda})}{\varepsilon}\notag\\
	&\hspace{0cm}= \lim_{\varepsilon\to0^{+}} \left( \frac{1}{\varepsilon}\tr(\bar{\Lambda} +\varepsilon\Delta\bar{\Lambda})-\frac{1}{\varepsilon} \tr(\bar{\Lambda})-\frac{1}{\varepsilon}\int \Psi_{\theta} \log \frac{G^{*}_{\theta} (\bar{\Lambda} +\varepsilon\Delta\bar{\Lambda})G_{\theta}}{G^{*}_{\theta} \bar{\Lambda} G_{\theta}} \right)\notag\\
	&\hspace{0cm}= -\lim_{\varepsilon\to0^{+}}\frac{1}{\varepsilon}\int \Psi_{\theta} \log \left(1+  \varepsilon\frac{G^{*}_{\theta}\Delta\bar{\Lambda}G_{\theta}}{G^{*}_{\theta} \bar{\Lambda} G_{\theta}}\right) \notag\\
	&\hspace{0cm}=-\int  \Psi_{\theta} \frac{G^{*}_{\theta}\Delta\bar{\Lambda}G_{\theta}}{G^{*}_{\theta}\bar{\Lambda}G_{\theta}},\notag
\end{align}
where we exploited the fact that $\tr(\Delta\bar{\Lambda})=0$, and, in the last step, the Taylor expansion of $\log(1+x)$. Eventually, since (i) $\bar{\Lambda}\in\N$, and (ii) $\Delta\bar{\Lambda}$ is such that $\bar{\Lambda}+\Delta\bar{\Lambda}\in\M$, there exists (at least) a frequency $\bar{\theta}\in[-\pi,\pi)$ such that $G^{*}_{\bar{\theta}}\bar{\Lambda}G_{\bar{\theta}}= 0$ and $G^{*}_{\bar{\theta}}\Delta\bar{\Lambda}G_{\bar{\theta}}\neq 0$. Since $\Psi_{\theta}>0$ for all $\theta\in[-\pi,\pi)$, this in turn implies
$
	\int  \Psi_{\theta} \frac{G^{*}_{\theta}\Delta\bar{\Lambda}G_{\theta}}{G^{*}_{\theta}\bar{\Lambda}G_{\theta}}=\infty,
$
which yields the thesis.
\end{proof}

\begin{Remark}\label{rem:N0}
Lemma \ref{lem:infcomp} provides a characterization of the elements of $\N$ in terms of directional derivatives of $\J(\cdot)$ along directions belonging to the subset of $\mathfrak{S}_{n}$ given by $\M$. Notice that this result typically does not characterize \emph{all} directional derivatives of $\J(\cdot)$ evaluated at elements in $\N$ along directions pointing to $\mathfrak{S}_{n}$. However, for a particular subset of $\N$ this is indeed the case. Let $x\in\C^{n}$, $\|x\|=1$, and let $P_{x}:=xx^{*}$ denote the orthogonal projection onto the subspace spanned by $x$. Consider the following subset of $\N$
\begin{align}
	\N_{0}:=&\{\,P_{\bar{x}}\in\mathfrak{S}_{n}\,:\,  \exists\, \bar{\theta}\in[-\pi,\pi), \, \bar{x}\in\C^{n},\, \|\bar{x}\|=1,  \text{ s.t. } \notag \\
	& \ \  \text{(i) } \langle\bar{x},G_{\bar{\theta}}\rangle=0, \text{ and }\notag \\
	& \ \  \text{(ii) } \langle x,G_{\bar{\theta}}\rangle \ne0,\, \forall\, x\in\C^{n},\, \|x\|=1,\, x\ne \bar{x} \,\}.\label{eq:N0}
\end{align}
By exploiting the same argument of the proof of  Lemma \ref{lem:infcomp}, it follows that the directional derivative of $\J(\cdot)$ evaluated at $\bar{\Lambda}\in\N_{0}$ along \emph{any} direction $\Delta\bar{\Lambda}\in\mathbb{H}_{n}\setminus\{0\}$ such that $\bar\Lambda+\Delta\bar{\Lambda}\in\mathfrak{S}_{n}$ is unbounded below. Moreover, it is worth noticing that, for the particular case $n=2$, it holds $\N_{0}\equiv \N$. \hfill $\diamondsuit$
%
\end{Remark}
}
To conclude this section, we recall the discrete-time version of LaSalle's invariance principle, whose proof can be found in \cite[Prop. 2.6]{LaSalle}.

\begin{Proposition}[Discrete-time LaSalle's invariance principle]\label{thm:LaSalle}
Consider a discrete-time system
\[
x(t+1) = f(x(t)), \quad x(0)\in \mathcal{X}, \ t\geq 0,
\]
where $f\colon\mathcal{X}\to\mathcal{X}$ is continuous and $\mathcal{X}$ is an invariant and compact set.
Suppose $V(\cdot)$ is a continuous function of $x\in \mathcal{X}$, bounded below and
satisfying
\[
\Delta V(x):=V(f(x))-V(x)\leq 0, \quad \forall x\in \mathcal{X},
\]
that is $V(x)$ is non-increasing along (forward) trajectories of the dynamics. Then any trajectory converges to the largest invariant subset $\mathcal{I}$ contained in $\mathcal{E} := \{\,x\in \mathcal{X}\, :\, \Delta V(x) = 0\,\}$.
\end{Proposition}

\section{Global convergence analysis}\label{sec:global}

In this section, we present the main results of this note. The first result (Theorem \ref{thm:non-inc}) states that the cost function \eqref{eq:PFcost} is always non-increasing along the trajectories of \eqref{eq:PFalg}. This result provides a positive answer to a conjecture raised in the conclusive part of \cite{FRT11}.

\begin{Theorem} \label{thm:non-inc}
For every $\Lambda\in\mathfrak{S}_{n}$ it holds
\begin{align}
\Delta\J(\Lambda):=\J(\Theta(\Lambda))-\J(\Lambda)\leq 0, \label{eq:ineqinit}
\end{align}
where $\J(\cdot)$ has been defined in \eqref{eq:PFcost}. Moreover $\Delta \J(\Lambda)= 0$ if and only if $\Theta(\Lambda)=\Lambda+\Lambda^{\perp}$, with $\Lambda^{\perp}\in(\mathrm{Range}\,\Gamma)^{\perp}$, where $\Gamma$ is the linear operator defined in Eq.~\eqref{eq:Gamma}.
\end{Theorem}

\begin{proof}
By plugging the expression of $\Theta(\cdot)$ into \eqref{eq:ineqinit}, we get
\begin{align}
&\Delta\J(\Lambda)=\J(\Theta(\Lambda))-\J(\Lambda)\leq 0 \notag \\
\Leftrightarrow\ &\int \Psi_{\theta} \log G_{\theta}^{*} \Theta(\Lambda)G_{\theta}-\int \Psi_{\theta} \log G_{\theta}^{*} \Lambda G_{\theta}\ge 0\notag\\
\Leftrightarrow\ &\int \Psi_{\theta} \log f_{\theta} \ge 0, \label{eq:equivalthm2}
\end{align}
where $f_{\theta}:=\frac{G^{*}_{\theta}\Theta(\Lambda)G_{\theta}}{G^{*}_{\theta}\Lambda G_{\theta}}$ is well-defined and strictly positive on $\T$, since $\Theta(\Lambda)$ has the same rank and kernel of $\Lambda\in\mathfrak{S}_{n}$, cf. \cite[Prop.~2.1]{FRT11}.

Firstly, we notice that the following inequality holds
\begin{align}\label{eq:ineq11}
  \int \Psi_{\theta} \log f_{\theta} = -2 \int \Psi_{\theta} \log f_{\theta}^{-1/2}\ge -2\log \int \Psi_{\theta} f_{\theta}^{-1/2},
\end{align}
which is a consequence of Lemma \ref{lem1}.

Secondly, by defining $\Pi_\theta:=\frac{\Lambda^{1/4}G_\theta G_\theta^*\Lambda^{1/4}}{G_\theta^*\Lambda G_\theta}$, we have the following chain of equations
\begin{align}
  1 & = \int \Psi_\theta f_\theta^{-1}f_\theta = \int \int f_\theta^{-1} \Psi_\theta \Psi_\omega \langle \Pi_\theta, \Pi_\omega\rangle \label{eq:ineq21}\\ 
  & \ge \int \int f_\theta^{-1/2}f_\omega^{-1/2}\Psi_\theta \Psi_\omega \langle \Pi_\theta, \Pi_\omega\rangle  \label{eq:ineq22}\\\ 
  & = \left\langle \int \Psi_\theta f_\theta^{-1/2}\Pi_\theta, \int\Psi_\omega f_\omega^{-1/2}\Pi_\omega \right\rangle  \notag\\\
  & = \|\Lambda^{1/2}\|_{\mathrm{F}}^2\left\|\int \Psi_\theta f_\theta^{-1/2}\Pi_\theta\right\|^2_{\mathrm{F}} \label{eq:ineq24}\\\
  & \ge \left|\left\langle\Lambda^{1/2},\int\Psi_\theta f_\theta^{-1/2}\Pi_\theta \right\rangle \right|^{2} \label{eq:ineq25}\\\
  & = \left(\int\Psi_\theta f_\theta^{-1/2}\right)^2 \label{eq:ineq26}
\end{align}
where
\begin{itemize}
    \item Eq.~\eqref{eq:ineq21} follows by noticing that $f_\theta=\int \Psi_\omega \langle \Pi_\theta, \Pi_\omega\rangle$,
    \item Eq.~\eqref{eq:ineq22} uses Lemma \ref{lem2} applied to the symmetric function $\Psi_\theta \Psi_\omega \langle \Pi_\theta, \Pi_\omega\rangle$,
    \item Eq.~\eqref{eq:ineq24} exploits the fact that $\|\Lambda^{1/2}\|_{\mathrm{F}}^2=\tr(\Lambda)=1$,
    \item Eq.~\eqref{eq:ineq25} follows from Cauchy--Schwarz inequality, and
    \item Eq.~\eqref{eq:ineq26} uses the fact that $\langle\Lambda^{1/2},\Pi_\theta\rangle=1$.
\end{itemize}
Eventually, a combination of the two sets of inequalities \eqref{eq:ineq11} and \eqref{eq:ineq21}-\eqref{eq:ineq26} yields $\int \Psi_{\theta} \log f_{\theta} \ge 0$ which, in turn, implies $ \Delta \J(\Lambda)\le 0$, in view of equivalence \eqref{eq:equivalthm2}.

Now it remains to prove that we attain equality in \eqref{eq:ineqinit} if only if $\Lambda\in\mathfrak{S}_{n}$ is such that $\Theta(\Lambda)=\Lambda+\Lambda^{\perp}$ with $\Lambda^{\perp}\in(\mathrm{Range}\,\Gamma)^{\perp}$. 
In view of the definition of $(\mathrm{Range}\,\Gamma)^{\perp}$ given in Eq.~\eqref{eq:rangeG}, the ``if'' part becomes straightforward. Indeed, if $\Theta(\Lambda)=\Lambda+\Lambda^{\perp}$, we have
\begin{align*}
	\Delta\J(\Lambda) &=  \int \Psi_{\theta} \log   \frac{G_{\theta}^{*} \Theta(\Lambda)G_{\theta}}{G_{\theta}^{*} \Lambda G_{\theta}} =  \int \Psi_{\theta} \log   \frac{G_{\theta}^{*} (\Lambda+\Lambda^{\perp})G_{\theta}}{G_{\theta}^{*} \Lambda G_{\theta}}\\ &=  \int \Psi_{\theta} \log   \frac{G_{\theta}^{*} \Lambda G_{\theta}}{G_{\theta}^{*} \Lambda G_{\theta}}=0.
\end{align*}
So it remains to prove the ``only if'' part, i.e., if equality in \eqref{eq:ineqinit} is attained for $\Lambda$ then $\Theta(\Lambda)=\Lambda+\Lambda^{\perp}$ with $\Lambda^{\perp}\in(\mathrm{Range}\,\Gamma)^{\perp}$. To this end, we notice that a necessary condition for \eqref{eq:ineqinit} to hold with equality is to have \eqref{eq:ineq11} satisfied with equality. By Lemma \ref{lem1}, this implies that the function $f_{\theta}$ is constant for every $\theta\in[-\pi,\pi)$, namely
\begin{align*}
	f_{\theta}&=\frac{G_{\theta}^{*} \Theta(\Lambda)G_{\theta}}{G_{\theta}^{*} \Lambda G_{\theta}}=\kappa,\quad \forall\, \theta\in[-\pi,\pi), 
\end{align*}
where $\kappa>0$ is a real constant. Now equality in \eqref{eq:ineqinit} is attained (if and) only if $\kappa=1$, and therefore we have that
\[
	G_{\theta}^{*} \Theta(\Lambda)G_{\theta} = G_{\theta}^{*} \Lambda G_{\theta},\quad \forall\, \theta\in[-\pi,\pi).
\]
From the latter equation and by definition of $(\mathrm{Range}\,\Gamma)^{\perp}$ in \eqref{eq:rangeG}, it follows that $\Theta(\Lambda) =\Lambda +\Lambda^{\perp}$, $\Lambda^{\perp}\in(\mathrm{Range}\,\Gamma)^{\perp}$. {
This completes the proof.}
\end{proof}

The following theorem is based on the previous one and states that iteration \eqref{eq:PFalg} always converges to the set of fixed points of $\Theta(\cdot)$.

\begin{Theorem} \label{thm:convergence}
  The trajectories generated by iteration \eqref{eq:PFalg} converge for all $\Lambda_{0}\in\mathfrak{S}_n$ to elements belonging to $\mathcal{F}:=\{\Lambda\in\mathfrak{S}_n\,:\, \Theta(\Lambda)=\Lambda\}$.
\end{Theorem}

\begin{proof}
  The proof consists of an application of the discrete-time version of LaSalle's invariance principle (Proposition \ref{thm:LaSalle}). The natural candidate Lyapunov function $V(\cdot)$ of Proposition \ref{thm:LaSalle} is given in this case by $\J(\Lambda)$ which is continuous and bounded for every $\Lambda\in\mathfrak{S}_n$ (Lemma \ref{lem:cont}), and, by virtue of Theorem \ref{thm:non-inc}, non-increasing along the (forward) trajectories of the dynamics \eqref{eq:PFalg}. Hence, by LaSalle's invariance principle, we have that the (forward) trajectories generated by iteration \eqref{eq:PFalg} converges to the largest invariant set $\mathcal{I}$ contained in 
  \[
    \mathcal{E}:=\{\,\Lambda\in \mathfrak{S}_n\,:\, \Delta\J(\Lambda)=0\,\}.
  \]

{\check
}

Therefore, it remains to show that the trajectories in $\mathcal{I}$ consist of fixed points of $\Theta(\cdot)$ only, that is, $\mathcal{I}\equiv\mathcal{F}$.
To this end, by Theorem \ref{thm:non-inc}, we know that the elements $\Lambda\in\mathcal{E}$ satisfy the condition
\begin{equation}\label{eq:constraint}
	\Theta(\Lambda)= \Lambda+\Lambda^{\perp},
\end{equation}
with $\Lambda^{\perp}\in(\mathrm{Range}\,\Gamma)^{\perp}$. In view of the latter constraint on the dynamics \eqref{eq:PFalg} and the definition of $(\mathrm{Range}\,\Gamma)^{\perp}$ in \eqref{eq:rangeG}, it follows that any trajectory belonging to $\mathcal{E}$ must obey to the recurrence relation
\begin{equation}\label{eq:iter2}
	\Lambda_{k+1}=\Lambda_{k}^{1/2} M \Lambda_{k}^{1/2}, \quad \Lambda_{0}\in\mathcal{E},\ k\geq 0,
\end{equation}
where
\[
 	M:=\int_{-\pi}^{\pi} \Psi_\theta\frac{G_\theta G^{*}_\theta}{G^{*}_\theta\Lambda_{0} G_\theta},
\]
depends on the initial condition $\Lambda_{0}$ only, in view of Eq.~\eqref{eq:constraint}. Now, since \eqref{eq:iter2} must generate unit trace trajectories starting from any $\Lambda_{0}\in\mathfrak{S}_n$, we have $\mathrm{tr}(\Lambda_{0})=\mathrm{tr}(\Lambda_{1})=\mathrm{tr}(\Lambda_{2})=1$. By exploiting the cyclic property and the linearity of the trace, this in turn implies that
\begin{align*}
&\mathrm{tr}(\Lambda_{0}) - 2\mathrm{tr}(\Lambda_{1}) + \mathrm{tr}(\Lambda_{2}) = \mathrm{tr}(\Lambda_{0}) - 2\mathrm{tr}(\Lambda_{0}^{1/2}M\Lambda_{0}^{1/2}) + \mathrm{tr}(M\Lambda_{0}^{1/2}M\Lambda_{0}^{1/2})\\
&=  \mathrm{tr}(\Lambda_{0}) - \mathrm{tr}(\Lambda_{0}^{3/4}M\Lambda_{0}^{1/4}) - \mathrm{tr}(\Lambda_{0}^{1/4}M\Lambda_{0}^{3/4}) +\mathrm{tr}(\Lambda_{0}^{1/4}M\Lambda_{0}^{1/2}M\Lambda_{0}^{1/4})\\
&=  \mathrm{tr}\left(\Lambda_{0}^{1/4}(I-M)\Lambda_{0}^{1/4}\right)^{2}= \left\|\Lambda_{0}^{1/4}(I-M)\Lambda_{0}^{1/4}\right\|_{\mathrm{F}}^{2} =0.
\end{align*} 
The previous equation is satisfied if and only if $\Lambda_{0}^{1/4}(I-M)\Lambda_{0}^{1/4}=0$, or, equivalently, if and only if
\[
	\Lambda_{0}=\Lambda_{0}^{1/2}M\Lambda_{0}^{1/2}.
\]
From the previous equation and Eq.~\eqref{eq:iter2} it readily follows that $\Lambda_{0}$ must be a fixed point of $\Theta(\cdot)$. This ends the proof.
\end{proof}

As a final result, we characterize a whole family of fixed points of $\Theta(\cdot)$ that are \emph{not} asymptotically stable. The following result provides a partial answer to another conjecture of \cite[Sec.~V]{FRT11} claiming that orthogonal rank-one projections which do not belong to the closure of the set of positive definite fixed points $\mathcal{P}$ are unstable equilibrium points of $\Theta(\cdot)$.
\begin{Proposition}
The set $\N_{0}$ defined in Eq.~\eqref{eq:N0} consists of fixed points that are not asymptotically stable for the dynamics \eqref{eq:PFalg}.
\end{Proposition}
\begin{proof}
Let $\bar{\Lambda} \in\N_{0}$. Notice that $\bar{\Lambda}$ is a rank-one orthogonal projection so that $\bar{\Lambda}$ is a fixed point of $\Theta(\cdot)$ by \cite[Prop.~4.3]{PF06}. Now observe that, in view of Lemma \ref{lem:infcomp} and Remark \ref{rem:N0}, all the (right-sided) directional derivatives at $\J(\bar{\Lambda})$ along directions pointing to $\mathfrak{S}_{n}$ take the value $-\infty$. This implies that in a sufficiently small neighbourhood $U$ of $\bar{\Lambda}$, it holds $\J(\bar{\Lambda})>\J(\Lambda)$, $\forall\,\Lambda\in U\cap \mathfrak{S}_{n}$, $\Lambda\ne \bar{\Lambda}$. In light of this, the claim follows from the fact that $\J(\cdot)$ is non-increasing along trajectories of the dynamics \eqref{eq:PFalg} (Theorem \ref{thm:non-inc}).
\end{proof}

{\check 

\section{Concluding remarks}\label{sec:conclusion}

In this paper, we analyzed the global convergence properties of an extensively studied fixed-point algorithm for the Kullback--Leibler approximation of spectral densities introduced by Pavon and Ferrante in \cite{PF06}. Our main result states that the Pavon--Ferrante algorithm globally converges to one of its fixed points. 

A question which remains unanswered in the paper concerns global convergence of the Pavon--Ferrante algorithm to a fixed point leading to the solution of the Georgiou--Lindquist spectral approximation problem, and, in particular, to the set $\mathcal{P}$ of \emph{positive definite} fixed points. A possible approach to guarantee convergence to a positive definite fixed point is to modify the Pavon--Ferrante iteration by adding at each step a suitable ``correction'' term belonging to $(\mathrm{Range}\, \Gamma)^{\perp}$ which prevents the iteration to approach the boundary of $\mathfrak{S}_{n}$. Notice in particular that, in view of Theorem~\ref{thm:non-inc}, the presence of such terms does not affect the decreasing behavior of $\J(\cdot)$ along the trajectories of the iteration, and, consequently, the convergence argument used in the proof of Theorem~\ref{thm:convergence}. This aspect will be the subject of future investigation.

}

\section{Acknowledgements}
The author wishes to thank Prof.~R. Sepulchre for having brought to his attention the problem addressed in the paper, and Prof.~A. Ferrante for several enlightening discussions.

\bibliographystyle{IEEEtran}
\bibliography{Bibliography}

\end{document}